\documentclass[a4paper,11pt]{article}
\usepackage[utf8x]{inputenc}

\usepackage[american]{babel}

\usepackage{amsmath} 
\usepackage{amsfonts} 
\usepackage{amsthm} 
\usepackage{epsf}  
\usepackage{appendix}
\usepackage{listings}
\usepackage{url}
\usepackage{amssymb,amsmath, amsthm, mathtools}
\usepackage{colordvi}
\usepackage{graphicx}
\usepackage{verbatim}
\usepackage{cancel}
\usepackage{caption}
\usepackage{epsf}
\usepackage{longtable}
\usepackage{multicol}
\usepackage{textcomp}
\usepackage[table]{xcolor}
\usepackage{mathdots}
\usepackage[all,  colour, line, dvips]{xy}
\usepackage{titlepic}
\usepackage{float}

\newtheorem{theorem}{Theorem}[section]
\newtheorem{lemma}[theorem]{Lemma}
\newtheorem{corollary}[theorem]{Corollary}
\newtheorem{proposition}[theorem]{Proposition}

\theoremstyle{definition}
\newtheorem{definition}[theorem]{Definition}
\newtheorem{example}[theorem]{Example}

\theoremstyle{remark}
\newtheorem*{remark}{Remark}

\newcommand{\SL}{\textrm{SL}}

\title{Low degree minimal generators of phylogenetic semigroups}
\author{Kaie Kubjas}
\date{}

\begin{document}

\maketitle 

\begin{abstract}

The phylogenetic semigroup on a graph generalizes the Jukes-Cantor binary model on a tree. Minimal generating sets of phylogenetic semigroups have been described for trivalent trees by Buczy{\'n}ska and Wi{\'s}niewski, and for trivalent graphs with first Betti number 1 by Buczy{\'n}ska.  We characterize degree two minimal generators of the phylogenetic semigroup on any trivalent graph. Moreover, for any graph with first Betti number 1 and for any trivalent graph with first Betti number 2 we describe the minimal generating set of its phylogenetic semigroup.
\end{abstract}

\section{Introduction}

Let $G$ be a graph. The phylogenetic semigroup on $G$ is a set of labelings of edges of $G$ by non-negative integers fulfilling some additional conditions. This set has naturally the structure of a semigroup by edge-wise addition.

The phylogenetic semigroup on a trivalent graph was defined by Buczy{\'n}ska~\cite{Buczynska12} as a generalization of the affine semigroup of the Jukes-Cantor binary model on a trivalent tree. Jukes-Cantor binary model is the simplest group-based model with the underlying group $\mathbb{Z}_2$. In~\cite{BBKM11}, Buczy{\'n}ska, Buczy{\'n}ski, Micha{\l}ek and the author further generalized the definition of the phylogenetic semigroup to arbitrary graphs. This definition agrees with Buczy{\'n}ska's definition for trivalent graphs. 

Besides phylogenetic algebraic geometry, phylogenetic semigroups appear in several other contexts.
In~\cite{JW92}, Jeffrey and Weitsman quantized the moduli space of flat $\textrm{SU}(2)$ connections on a two-dimensional surface of genus $g$ using a real polarization. The dimension of the quantization is counted by integral fibers of the polarization, which are in one-to-one correspondence with the labelings of a trivalent graph $G$ with first Betti number $g$ that satisfy the quantum Clebsch-Gordan conditions. These labelings are exactly the elements of the phylogenetic semigroup on $G$. Moreover, the number of labelings that satisfy the quantum Clebsch-Gordan conditions matches the Verlinde formula for the $\textrm{SU}(2)$ Wess-Zumino-Witten model in the quantum field theory~\cite{Verlinde88}.

In more recent work, Sturmfels and Xu~\cite{SX10} showed that the projective coordinate ring of the Jukes-Cantor binary model is a sagbi degeneration of the Cox ring of the blow-up of $\mathbb{P}^{n+3}$ at $n$ general points. Manon generalized their construction showing that the algebra of $\SL_2(\mathbb{C})$ conformal blocks for a stable curve of genus $g$ with $n$ marked points flatly degenerates to the semigroup algebra of the phylogenetic semigroup on a graph with first Betti number $g$ with $n$ leaves~\cite{Manon09}. 

Low degree minimal generators  of phylogenetic semigroups have been previously studied for trees and graphs with first Betti number 1. Phylogenetic semigroups on trees are generated by degree one labelings, known as networks~\cite{BW07,DBM12}. Buczy{\'n}ska studied minimal generators of phylogenetic semigroups on trivalent graphs with first Betti number 1. She proved that any minimal generator of the phylogenetic semigroup on a trivalent graph with first Betti number 1 has degree at most two, and explicitly described minimal generating sets~\cite{Buczynska12}. 

We extend this result from trivalent graphs to general graphs, i.e. we describe the minimal generating set of the phylogenetic semigroup on any graph with first Betti number $g\leq 1$. Moreover, we characterize degree two minimal generators on trivalent graphs with first Betti number $g>1$.

We also specify the bound on the maximal degree of the minimal generating set for graphs with first Betti number 2. By~\cite{BBKM11}, the maximal degree of the minimal generating set of the phylogenetic semigroup on a graph with first Betti number 2 is at most three. We explicitly characterize when  the maximal degree three is attained, and when the maximal degree is equal to two or one. If the degree three is attained, we describe the degree three minimal generators.

Finally, we list maximal degrees of minimal generating sets of phylogenetic semigroups on some graphs with first Betti number 3, 4 or 5. We speculate that the maximal degree depends on the separateness of the cycles of the graph. Having low maximal degree is especially interesting from the perspective of SL$_2(\mathbb{C})$ conformal block algebras as this ensures low maximal degree for the minimal generators of these algebras, see~\cite{Manon_bounds_on_generators}.

In Section~\ref{section:phylogenetic_semigroups}, we introduce basics about phylogenetic semigroups on graphs. In Section~\ref{section:first_betti_number_one}, we give a shortened proof of Buczy{\'n}ska's theorem about minimal generators of phylogenetic semigroups on trivalent graphs with first Betti number 1, and we generalize the statement to general graphs with first Betti number 1. In Section~\ref{section:degree_two_generators}, we characterize degree two minimal generators on an arbitrary trivalent graph. In Section~\ref{section:first_betti_number_two}, we study the explicit maximal degree of the minimal generating set of the phylogenetic semigroup on a graph with first Betti number 2. In Section~\ref{section:degree_three_generators}, we describe minimal generating sets of phylogenetic semigroups on trivalent graphs with first Betti number 2. In the last section, we list examples of these maximal degrees for graphs with first Betti numbers 3,4 and 5.

\section{Phylogenetic semigroups}\label{section:phylogenetic_semigroups}

In this section, we define phylogenetic semigroups on graphs as in~\cite{BBKM11} and recall some basic properties about these semigroups.

\begin{definition}
Let $G$ be a graph. A \textit{path} in $G$ is a sequence of unrepeated edges which connect a sequence of vertices. Moreover, we require the first and the last vertex to be either both leaves or equal. In the latter case, a path is called a \textit{cycle}.  A \textit{network} is a disjoint union of paths. A \textit{cycle edge} is an edge on a cycle of $G$. A \textit{cycle leg} is an edge incident to a cycle edge, but is not a cycle edge.  We denote the \textit{disjoint sum} of graphs $G_1$ and $G_2$ by $G_1\sqcup G_2$. We denote by $G^e$ the graph obtained from $G$ by cutting an internal edge $e$. More specifically, \textit{cutting an internal edge} $e$ means replacing $e$ by two leaf edges $e_1$ and $e_2$ where $\partial_1(e_1)=\partial_1(e)$ and $\partial_1(e_2)=\partial_2(e)$. Here $\partial_1(e), \partial_2(e)$ denote endpoints of an edge $e$.
\end{definition}

\begin{definition}
Let $T$ be a tree with the set of edges $E$ and the set of inner vertices $I$.  Define lattices
\begin{displaymath}
 L_T=\{x\in \mathbb{Z}^E:\sum _{v\in e}x_e \in 2\mathbb{Z} \textrm{ for every } v\in I\}
\end{displaymath}
and
\begin{displaymath}
L_T^{gr}=L_T\oplus \mathbb{Z} 
\end{displaymath}
together with the degree map
\begin{displaymath}
\textrm{deg}:L_T^{gr}=L_T\oplus \mathbb{Z}\rightarrow \mathbb{Z}
\end{displaymath}
given by the projection on the last summand. 
\end{definition}

\begin{definition}
The lattice polytope associated with the Jukes-Cantor binary model on $T$ is
\begin{displaymath}
 P_T=\textrm{conv}\{x\in L_T:x_e\in\{0,1\} \textrm{ for every } e\in E\}.
\end{displaymath}
The phylogenetic semigroup $\tau(T)$ on $T$ is
\begin{displaymath}
 \tau(T)=\textrm{cone}(P_T\times \{1\})\cap L_T^{gr}.
\end{displaymath}
\end{definition}

\begin{definition}\label{defin_graph_path_network}
\textit{First Betti number} of a graph is the minimal number of cuts that would make the graph into a tree.
\end{definition}

Given a graph $G$, we associate a tree $T$ with a set of distinguished pairs of leaves to $G$ and define the phylogenetic semigroup on $G$ using the phylogenetic semigroup on $T$. We construct the tree $T$ inductively on first Betti number $g$ of $G$. If $g=0$, then $G$ itself is the associated tree. If $g>0$, then we replace a cycle edge $\overline{e}$ by two leaf edges $e'$ and $e''$ where $\partial_1(e')=\partial_1(\overline{e})$ and $\partial_1(e'')=\partial_2(\overline{e})$. This replacement gives a graph with first Betti number $g-1$ and a distinguished pair of leaves $(e',e'')$. Doing this procedure $g$ times gives a tree $T$ and $g$ distinguished pairs of leaves. Although the tree $T$ and the set of distinguished pairs of leaves are in general not unique, the phylogenetic semigroup on $G$ does not depend on the choices we make.

\begin{definition}
Let $G$ be a graph.
Let $T$ be the associated tree with a set of distinguished pairs of leaves $\{(e_i',e_i'')\}$.
We define the \textit{phylogenetic semigroup on $G$} as
$$\tau(G)=\tau(T)\cap \bigcap_{i}(x_{e_i'}=x_{e_i''}).$$
In other words, $\tau(G)$ consists of those labelings of $\tau(T)$ where the label on $e_i'$ is identical to the one on $e_i''$,
and thus the labeling  of $T$ gives a labeling of $G$.
Similarly, define the lattice
\begin{displaymath}
L_G^{gr}=L_T^{gr}\cap \bigcap_{i}(x_{e_i'}=x_{e_i''})
\end{displaymath}
together with the degree map induced by the degree map of $L_T^{gr}$.
\end{definition}

The phylogenetic semigroup $\tau(G)$ has a unique minimal generating set. We call the elements of the minimal generating set minimal generators, or sometimes also indecomposable elements of $\tau(G)$.

As we often deal with trivalent graphs, we introduce notation specific to these graphs. Let $G$ be a trivalent graph and $v$ be an inner vertex of $G$. Let $\{e_1, e_2, e_3\}$ be the edges of the tripod and
$i_v$ a map that is locally an embedding of the tripod into $G$ and sends the central vertex of the tripod
to $v$.
For $\omega \in L_G^{gr}$ denote
\[
\begin{array}{lcr}
a_v(\omega):=\omega_{i_v(e_1)},  &
b_v(\omega):=\omega_{i_v(e_2)},  &
c_v(\omega):=\omega_{i_v(e_3)}.  \\
\end{array}
\]
In other words, $a_v, b_v, c_v$ measure the coefficients of $\omega$
at the edges incident to $v$.

\begin{definition}
The \textit{degree} of $\omega\in L_G^{gr}$ \textit{at an inner vertex} $v\in I$ is
\[
\deg_v(\omega):=\frac{1}{2}\bigl(a_v(\omega)+b_v(\omega)+c_v(\omega)\bigr).
\]
\end{definition}

\begin{lemma}[\cite{Buczynska12}, Definition~2.18 and Lemma~2.23]\label{def_cone_of_G}
For a trivalent graph $G$ the \textit{phylogenetic semigroup $\tau(G)$ on $G$} is the set of elements $\omega$ satisfying the following conditions
\begin{enumerate}
\renewcommand{\theenumi}{\textnormal{[$\heartsuit\!\!\heartsuit$]}}
\item  \label{item_parity_condition} parity condition:  $\omega \in L_G^{gr}$,
\renewcommand{\theenumi}{\textnormal{[+]}}
\item  \label{item_non-negativity_condition} non-negativity condition:
$\omega_e \ge 0$ for any $e \in E$,
\renewcommand{\theenumi}{\textnormal{[$\triangle$]}}
 \item  \label{item_triangle_inequalities} triangle inequalities:
$ |a_v(\omega)-b_v(\omega)| \leq c_v(\omega) \leq a_v(\omega)+ b_v(\omega)$, for each inner vertex $v\in I$,
\renewcommand{\theenumi}{\textnormal{[\textdegree]}}
\item  \label{item_degree_inequality}
degree inequalities:  $\deg(\omega) \geq \deg_v(\omega)$ for any $v\in I$.
\end{enumerate}
\end{lemma}

Let $G$ be a graph and $\overline{e}$ an inner edge of $G$. Let $e'$ and $e''$ be new leaf edges obtained by cutting $G$ at $\overline{e}$, as illustrated in Figure~\ref{figure:cutting_an_edge}.
Then $\omega \in \tau(G)$ gives an element $\overline{\omega} \in \tau(G^{\overline{e}})$:
\begin{displaymath}
\overline{\omega}_e=\left\{
\begin{array}{ll}
 \omega _e & \textrm{if } e\notin \{e',e''\},\\
 \omega _{\overline{e}} & \textrm{if } e\in \{e',e''\}.
\end{array}\right.
\end{displaymath}
On the contrary, given $\overline{\omega} \in \tau (G^{\overline{e}})$, it gives an element $\omega \in \tau(G)$ if and only if $\overline{\omega}_{e'}=\overline{\omega}_{e''}$:
\begin{displaymath}
\omega_e=\left\{
\begin{array}{ll}
 \overline{\omega} _e & \textrm{if } e\neq \overline{e},\\
 \overline{\omega} _{e'} & \textrm{if } e= \overline{e}.
\end{array}\right.
\end{displaymath}

\begin{figure}[ht]
\centering
\epsfxsize=240pt\epsfbox{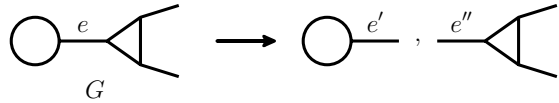} 
\caption{A graph $G^{\overline{e}}$ obtained by cutting an inner edge $\overline{e}$ of $G$}
\label{figure:cutting_an_edge}
\end{figure}

In~\cite{Buczynska12}, a polygon graph $G$ was defined as a graph with $2k$ edges, $k$ of which form the only cycle of $G$ and the remaining $k$ edges are cycle legs. The use of polygon graphs simplifies the study of phylogenetic semigroups on trivalent graphs with first Betti number 1. We generalize this definition to be able to simplify the study of phylogenetic semigroups on any graph.

\begin{definition}
A graph $G$ with first Betti number $g\geq 1$ is called a \textit{multiple polygon graph} if for no edge $e$ we can write $G^e=G'\sqcup G''$ with $G'$ or $G''$ a tree with more than one edge, see Figure~\ref{figure:graphs_with_first_betti_number_2_cases} for examples. A multiple polygon graph is a \textit{polygon graph} if it has first Betti number 1. 
\end{definition}

\begin{lemma}
Given a graph $G$ with first Betti number $g\geq 1$, there exist non-cycle inner edges $e_1,\ldots ,e_k$ of $G$ such that  $G^{e_1,\ldots ,e_k}=G_0\sqcup G_1 \sqcup \ldots \sqcup G_{k}$ where $G_0$ is a multiple polygon graph and $G_1,\ldots ,G_k$ are trees.
\end{lemma}

\begin{proof}
Choose all non-cycle edges $e$ such that we can write $G^e=G'\sqcup G''$ with $G''$ a tree with more than one edge and $e$ maximal with this property, i.e. there is an edge $\overline{e}$ incident to $e$ such that we cannot write $G^{\overline{e}}={G}'\sqcup {G}''$ with $G'$ or $G''$ a tree with more than one edge.
\end{proof}

\begin{lemma}\label{lemma:graph_union_tree}
Let $G$ be a graph and $\omega \in \tau(G)$. Let $e$ be a non-cycle inner edge such that $G^e=G'\sqcup G''$ with $G''$ a tree. Then any decomposition of $\omega|_{G'}\in \tau(G')$ lifts to a decomposition of $\omega\in \tau(G)$. 
\end{lemma}

\begin{proof}
This lemma is stated for trivalent graphs in~\cite[Lemma~2.31]{Buczynska12}. Since $\tau(T)$ is normal for any tree $T$~\cite[Proposition 18]{DBM12}, then the proof works for the general case exactly the same way as it does for the trivalent case.
\end{proof}

\begin{corollary}\label{lemma:polygon_graph}
Let $G$ be a graph and $\omega \in \tau(G)$. Let $e_1,\ldots ,e_k$ be  non-cycle inner edges such that  $G^{e_1,\ldots ,e_k}=G_0\sqcup G_1 \sqcup \ldots \sqcup G_{k}$ where $G_0$ is a multiple polygon graph and $G_1,\ldots ,G_k$ are trees. Then any decomposition of $\omega|_{G_0}\in \tau(G_0)$ lifts to a decomposition of $\omega\in \tau(G)$. 
\end{corollary}

\begin{proof}
We can use Lemma~\ref{lemma:graph_union_tree} iteratively. 
\end{proof}

\section{Graphs with First Betti Number 1}\label{section:first_betti_number_one}

In this section, we study minimal generating sets of phylogenetic semigroups on graphs with first Betti number 1. Buczy{\'n}ska did this for trivalent graphs~\cite{Buczynska12}. We give a shortened proof of her result, and as a corollary describe the minimal generating set of the phylogenetic semigroup on any graph with first Betti number 1.

Let $G$ be a graph. Networks can be seen as degree one elements of $\tau(G)$. We define $\omega$ corresponding to a network $\Gamma$ in the following way: 
\begin{eqnarray*}
\omega _e=1 & & \quad   \textrm{if } e \textrm{ belongs to } \Gamma,\\
\omega _e=0 & & \quad \textrm{otherwise}.  
\end{eqnarray*}
It follows from the definition of a network that the parity condition is fulfilled for $\omega $ at every inner vertex of $G$. Hence $\omega \in \tau(G)$. We will often use the notion network for the corresponding labeling $\omega \in \tau(G)$.

It has been shown for various classes of graphs that networks are in one-to-one correspondence with degree one elements of a phylogenetic semigroup~\cite[Lemma~2.3]{BW07}, ~\cite[Lemma~2.26]{Buczynska12}. For an arbitrary tree this was stated in~\cite[Section 2]{BBKM11}, but no proof was given. We did not find proofs for arbitrary trees or graphs in the literature and therefore will present them here.

\begin{lemma}\label{lemma:networks_on_trees}
Let $T$ be a tree. There is one-to-one correspondence between networks and degree one elements of $\tau(T)$.
\end{lemma}

\begin{proof}
We will prove the lemma by induction on the number of inner vertices of $T$.

\noindent Base case: The statement of the lemma clearly holds for claw trees. 

\noindent Induction step: Let $T$ be a tree with $n>1$ inner vertices and $\omega \in \tau(T)$ a degree one labeling. If $T$ has more than one connected component, then by induction $\omega$ restricted to any connected component is a disjoint union of paths. Hence $\omega$ is a disjoint union of paths. 

If $T$ has one connected component, let $e$ be an inner edge, $e_1,e_2$ new leaf edges obtained by cutting $T$ at $e$ and write $T^e=T_1\sqcup T_2$. Then $\omega $ restricted to either tree is a disjoint union of paths. If $\omega_e=0$, then $\omega$ is the disjoint union of exactly the same paths. If $\omega _e=1$, then the path of $T_1$ containing $e_1$ and the path of $T_2$ containing $e_2$ are combined to one path of $T$ containing $e$. Hence $\omega $ is a disjoint union of paths. 
\end{proof}

\begin{lemma}\label{lemma:networks_on_graphs}
Let $G$ be a graph. There is one-to-one correspondence between networks and degree one elements of $\tau(G)$.
\end{lemma}

\begin{proof}
We will prove the lemma by induction on first Betti number $g$ of $G$.

\noindent Base case: The statement of the lemma holds for trees by Lemma~\ref{lemma:networks_on_trees}.

\noindent Induction step: Let $G$ be a graph with first Betti number $g>1$ and $\omega \in \tau(G)$ a degree one labeling. Let $e$ be a cycle edge of $G$ and $e_1,e_2$ new leaf edges obtained by cutting $G$ at $e$. The graph $G^e$ has first Betti number $g-1$. Then $\omega $ gives $\overline{\omega}\in G^e$ that is a disjoint union of paths containing both $e_1,e_2$ or neither of them. If $\omega_e=0$, then $\omega$ is the disjoint union of exactly the same paths. If $\omega _e=1$, then there are two possibilities. Either there is a path in $\overline{\omega}$ with first edge $e_1$ and last edge $e_2$ which lifts to a cycle in $\omega$. Or there is a path in $\overline{\omega}$ with first edge $e'$ and last edge $e_1$, and another path in $\overline{\omega}$ with first edge $e_2$ and last edge $e''$, where $e',e''$ are leaf edges. These paths in $\overline{\omega}$ lift to a single path in $\omega$ with first edge $e$ and last edge $e'$ in $G$. 
\end{proof}

\begin{corollary}
Let $G$ be a graph. All networks are included in the minimal generating set of $\tau(G)$. 
\end{corollary}

\begin{proof}
For any graded affine semigroup $\mathbb{N}\mathcal{A}$ all degree one minimal generators are included in the minimal generating set of $\mathbb{N}\mathcal{A}$. 
\end{proof}

\begin{theorem}[\cite{Buczynska12}, Theorem~2.29]\label{thm:buczynska}
Let $G$ be a trivalent graph with first Betti number 1 and $\omega \in \tau(G)$. Then $\omega$ is a minimal generator of $\tau(G)$ if and only if it satisfies one of the following conditions:
\begin{itemize}
 \item $\omega$ is a network, or

\item $\omega$ has degree two, and satisfies the following three conditions
\begin{itemize}
\item[(i)] $\omega _e=1$, for all cycle edges $e$,

\item[(ii)] $\omega _e=2$, for an odd number of cycle legs,

\item[(iii)] $\omega _e=0$, for the remaining cycle legs.
\end{itemize}
\end{itemize}

\end{theorem}

We give a shortened proof of this theorem. The following lemma will be an important part of it.

\begin{lemma}\label{lemma:decomposition_0_or_d}
Let $G$ be a graph with first Betti number 1. Let $\omega \in \tau(G)$ be of degree $d$. If there is a cycle edge $e$ with $\omega _e=0$ or $\omega _e=d$, then $\omega $ decomposes as a sum of degree one elements.
\end{lemma}

\begin{proof}
Let $e$ be a cycle edge and $e_1,e_2$ new leaf edges obtained by cutting $G$ at $e$. Notice that $G^e$ is a tree. Then $\omega$ gives $\overline{\omega} \in \tau (G^e)$ that decomposes into degree one elements $\overline{\omega} = \overline{\omega _1}+\ldots +\overline{\omega _d}$. Since $(\overline{\omega _i})_{e_1}=(\overline{\omega _i})_{e_2}$ for all $i$, the decomposition $\overline{\omega} = \overline{\omega _1}+\ldots +\overline{\omega _d}$ gives a decomposition $\omega = \omega _1+\ldots +\omega _d$ of $\omega\in \tau(G)$.
\end{proof}

\begin{proof}[Proof of Theorem~\ref{thm:buczynska}]
By Corollary~\ref{lemma:polygon_graph} we can assume that $G$ is a trivalent polygon graph. First we prove that any minimal generator of $\tau(G)$ has degree at most two. Let $\omega \in \tau(G)$ be of degree $d$. Let $e$ be a cycle edge and $e_1,e_2$ new leaf edges obtained by cutting $G$ at $e$. Then $\omega$ gives $\overline{\omega} \in \tau (G^e)$ that decomposes as a sum of degree one elements $\overline{\omega} = \overline{\omega _1}+\ldots +\overline{\omega _d}$. If $(\overline{\omega _i})_{e_1}=(\overline{\omega _i})_{e_2}$ then $\overline{\omega _i}$ gives an element $\omega_i\in \tau (G)$. Otherwise there exists $j$ such that $(\overline{\omega _i})_{e_1}=(\overline{\omega _j})_{e_2}$ and $(\overline{\omega _j})_{e_1}=(\overline{\omega _i})_{e_2}$, because $\overline{\omega}_{e_1}=\overline{\omega}_{e_2}$. Thus $\overline{\omega _i}+\overline{\omega _j}$ gives a degree two element $\omega_i+\omega_j \in \tau (G)$.

Degree one elements of $\tau (G)$ are networks by Corollary~\ref{lemma:networks_on_graphs}. By Lemma~\ref{lemma:decomposition_0_or_d}, all degree two indecomposable elements $\omega $ have $\omega _e=1$ on all cycle edges $e$. Since $G$ is a trivalent graph, we have $\omega_e\in \{0,2\}$ for all cycle legs because of the parity condition. Assume $\omega _e=2$ for an even number of cycle legs $e_1,\ldots e_{2k}$ in clockwise order. Denote by $P_i$ the path starting at $e_i$ and ending at $e_{i+1}$ (at $e_0$ for $i=2k$). Then $\omega$ decomposes as the sum of networks $P_1\cup P_3 \cup \ldots \cup P_{2k-1}$ and $P_2 \cup P_4 \cup \ldots \cup P_{2k}$. Hence for $\omega$ indecomposable $\omega _e=2$ for an odd number of cycle legs.

Conversely, assume that $\omega \in \tau(G)$ has degree two and fulfills $(i),(ii),(iii)$. Suppose $\omega=\omega _1 + \omega _2$, where $\omega _1,\omega _2$ are networks. For all cycle legs $e$ with $\omega_e=2$ we have $(\omega _i)_e=1$, since $(\omega _i)_e\leq 1$ for all edges $e$. Hence $(\omega _i)_e=1$ for odd number of leaves of $G$. But this is contradiction to the fact that $\omega _i$ is a network. 
\end{proof}

\begin{remark}
We know from~\cite{Buczynska12,BBKM11} that a minimal generator of the phylogenetic semigroup on a graph with first Betti number 1 has degree at most two. We showed this above to give a simple and self-containing proof.
\end{remark}

\begin{corollary}\label{cor:first_betti_number_1}
Let $G$ be a graph with first Betti number 1 and $\omega \in \tau(G)$. Then $\omega$ is a minimal generator of $\tau(G)$ if and only if it satisfies one of the following conditions:
\begin{itemize}
 \item $\omega$  is a network, or
\item $\omega$ has degree two, and satisfies the following three conditions
\begin{itemize}
\item[(i)] $\omega_e=1$ for all cycle edges $e$,
\item[(ii)] $\omega _e=2$, for an odd number of cycle legs,
\item[(iii)] $\omega _e=0$, for the remaining cycle legs.
\end{itemize}
\end{itemize} 
\end{corollary}

\begin{proof}
Let $G'$ be a trivalent graph constructed from $G$ in the following way: Replace all vertices $v$ with valency higher than three by two new vertices $v'$ and $v''$ together with a new edge between them, let two edges incident to $v$ be incident to $v'$ and the rest of the edges incident to $v$ be incident to $v''$. Moreover, if $v$ is on the cycle, let one cycle edge incident to $v$ be incident to $v'$ and let the other cycle edge incident to $v$ be incident to $v''$. This assures that we do not add any cycle legs. After a finite number of replacements we get a trivalent graph $G'$. As in~\cite[Lemma 4.1]{BBKM11}, $\tau(G)$ is the coordinate projection of $\tau(G')$ that forgets coordinates corresponding to new edges. In particular, if $\omega'\in\tau(G')$ is decomposable, then its projection in $\tau(G)$ is also decomposable.

By~\cite{BBKM11}, any minimal generator of $\tau(G)$ has degree at most two. Degree one elements are networks. We are left with describing degree two indecomposable elements of $\tau(G)$. A degree two indecomposable element $\omega \in \tau(G)$ is the coordinate projection of a degree two indecomposable element of $\tau(G')$. Since all cycle legs of $G'$ are also cycle legs of $G$, then by Theorem~\ref{thm:buczynska} the conditions $(i),(ii),(iii)$ are fulfilled for $\omega$.

Conversely, assume that the conditions $(i),(ii),(iii)$ are fulfilled. Suppose $\omega=\omega _1 + \omega _2$, where $\omega _1,\omega _2$ are networks. For all cycle legs $e$ with $\omega_e=2$ we have $(\omega _i)_e=1$, since $(\omega _i)_e\leq 1$ for all edges $e$. Hence $(\omega _i)_e=1$ for odd number of leaves of $G'$. But this is a contradiction to the fact that $\omega _i$ is a network. 
\end{proof}

\section{Degree Two Minimal Generators}\label{section:degree_two_generators}

In this section, we describe degree two indecomposable labelings for any trivalent graph $G$.

\begin{lemma}\label{cor:degree_2_generators}
Let $G$ be any graph and $\omega \in \tau(G)$ a degree two labeling. If there exists a cycle $G'$ of $G$ such that $\omega_e=1$ for all cycle edges $e\in G'$, $\omega_e=2$ for an odd number cycle legs $e$ of $G'$ and $\omega_e=0$ for the remaining cycle legs $e\in G'$, then the labeling $\omega$ is indecomposable.
\end{lemma}

\begin{proof}
If $\omega$ decomposes, then a decomposition of $\omega$ restricts to a decomposition of  $\omega |_{G'}\in \tau(G')$, where $G'$ is a cycle together with its cycle legs. Thus the statement follows from Corollary~\ref{cor:first_betti_number_1}.
\end{proof}

\begin{proposition}\label{prop:degree_2_generators}
Let $G$ be a trivalent graph and $\omega \in \tau(G)$ a degree two labeling. The labeling $\omega$ is indecomposable if and only if there exists a cycle $G'$ of $G$ together with its cycle legs such that $\omega |_{G'}\in \tau(G')$ is indecomposable.
\end{proposition}

\begin{proof}
One direction follows from Lemma~\ref{cor:degree_2_generators}. We show by induction on first Betti number of $G$ that if $\omega \in \tau(G)$ is a degree two indecomposable labeling then there exists a cycle  $G'$ together with its cycle legs such that $\omega|_{G'}$ is a degree two indecomposable labeling.

Base case: If first Betti number of $G$ is 1, then the statement follows from Theorem~\ref{thm:buczynska}.

Induction step: Assume that first Betti number of $G$ is $g>1$. If more than one connected component of $G$ contains a cycle, then there exists a connected component $C$ of $G$ containing a cycle such that $\omega |_{C}\in \tau(C)$ is an indecomposable element of degree two. Since first Betti number of $C$ is less than $g$, we know by induction that there exists a cycle $G'$ of $C$ together with its cycle legs such that $\omega |_{G'}\in \tau(G')$ is an indecomposable element of degree two.

Otherwise all cycles of $G$ live in the same connected component of $G$. If $\omega_e=1$ for all cycle edges $e$, then by the parity condition  $\omega_e\in\{0,2\}$ for all cycle legs $e$. In particular, none of the cycle legs is simultaneously a cycle edge and there exists a cycle leg $e$ that separates some cycles of $G$. Let $e_1,e_2$ be the new leaf edges obtained by cutting $G$ at $e$ and write $G^e=G_1\sqcup G_2$. Then $\omega$ gives $\omega_1 \in \tau(G_1)$ and $\omega_2 \in \tau(G_2)$ with at least one of them indecomposable, otherwise one could lift these decompositions to a decomposition of $\omega$. By induction, for $i$ with $\omega_i$ indecomposable there exists a cycle $G'$ of $G_i$ together with its cycle legs such that $\omega_i |_{G'}\in \tau(G')$ is indecomposable. Thus $\omega|_{G'} \in \tau(G')$ is indecomposable. 

If there exists a cycle edge $e$ with $\omega_e\in \{0,2\}$, then let $e_1$ and $e_2$ be new leaf edges obtained by cutting  $G$ at $e$. The labeling $\omega\in \tau(G)$ gives a labeling $\overline{\omega}\in \tau(G^e)$ that is indecomposable. Otherwise one could lift a decomposition $\overline{\omega}=\overline{\omega}_1+\overline{\omega}_2$ to a decomposition $\omega=\omega_1+\omega_2$, because $(\overline{\omega_i})_{e_1}=(\overline{\omega_i})_{e_2}$. The graph $G^e$ has first Betti number less than $g$. By induction, there exists a cycle $G'$ of $G^e$ together with cycle legs  such that $\overline{\omega}|_{G'}\in \tau(G')$ is indecomposable. Thus $\omega|_{G'} \in \tau(G')$ is indecomposable.
\end{proof}

\section{Graphs with First Betti Number 2}\label{section:first_betti_number_two}

We know from~\cite{BBKM11} that any minimal generator of the phylogenetic semigroup on a graph with first Betti number 2 has degree at most three. In this section, we will explicitly describe which phylogenetic semigroups have which maximal degrees of minimal generating sets for graphs with first Betti number 2. We will see that there are graphs with maximal degrees of minimal generators equal to one, two and three.  Our analysis is based on five different cases depending on the structure of the graph - whether the cycles live in different components of the graph, share at least one edge, share exactly a single vertex, there is a single edge connecting the cycles, or the cycles are more than one edge apart from each other, see Figure~\ref{figure:graphs_with_first_betti_number_2_cases} for latter four cases. 

\begin{figure}[ht]
\centering
\epsfxsize=360pt\epsfbox{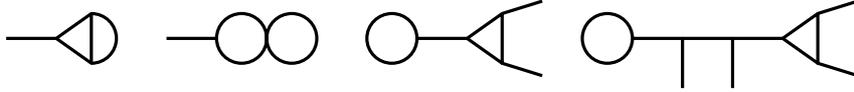}  
\caption{A graph with   (a) cycles sharing at least one edge, (b) cycles sharing exactly a single vertex, (c) a single edge connecting cycles, (d) cycles more than one edge apart from each other}
\label{figure:graphs_with_first_betti_number_2_cases}
\end{figure}

\begin{remark}
Assume a graph has a degree two vertex $v$. Denote the edges incident to $v$ by $e_1$ and $e_2$. By the definition of the phylogenetic semigroup on a graph,  we have $\omega _{e_1}=\omega _{e_2}$ for $\omega \in \tau(G)$. Hence elements of $\tau(G)$ are in one-to-one correspondence with elements of $\tau(G')$, where $G'$ is obtained from $G$ by replacing $e_1$ and $e_2$ by a single edge. To simplify future analysis, from now on  we will assume that graphs posses no degree two vertices.
\end{remark}

\begin{theorem}\label{thm:graphs_with_first_betti_number_2}
Let $G$ be a graph with first Betti number 2. The maximal degree of a minimal generator of $\tau(G)$ is 
\begin{itemize}
\item one if and only if $G$ does not contain any cycle legs that are not cycle edges;

\item two if and only if $G$  the cycles of $G$ live in different connected components, or $G$ contains at least one cycle leg that is not a cycle edge, all cycles of $G$ live in the same connected component and are not separated by an inner vertex;

\item three if and only if the minimal cycles of $G$ live in the same connected component and are separated by at least one inner vertex;
\end{itemize}
\end{theorem}

We will study these different cases in Lemmas~\ref{lemma:graph1}--\ref{lemma:graph5}.

\begin{figure}[ht]
\centering
\epsfxsize=120pt\epsfbox{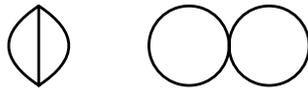} 
\caption{Graphs with first Betti number 2 whose phylogenetic semigroups are normal}
\label{figure:graph_with_three_edges}
\end{figure}

\begin{lemma}\label{lemma:graph1}
Let $G$ be a graph with first Betti number 2 that does not contain any cycle legs that are not cycle edges.
The maximal degree of a minimal generator of $\tau(G)$ is one.
\end{lemma}

\begin{proof}
The cycles of $G$ live in the same connected component. Otherwise $G$ would have a degree two vertex. If the connected component of  $G$ containing the cycles has one vertex, then it is isomorphic to the right graph in Figure~\ref{figure:graph_with_three_edges}. If the connected component of  $G$ containing the cycles has two vertices, then it is isomorphic to the left graph in Figure~\ref{figure:graph_with_three_edges}. The connected component of $G$ containing the cycles cannot have three or more vertices, because every vertex must belong to at least two cycles.

By computations with \texttt{Normaliz}~\cite{normaliz}, the phylogenetic semigroup of the left graph in Figure~\ref{figure:graph_with_three_edges} is
\begin{displaymath}
\mathbb{N} \{(1,0,0,0),(1,1,1,0),(1,1,0,1),(1,0,1,1)\},
\end{displaymath}
where the first coordinate corresponds to the degree and the other three coordinates correspond to edges of $G$ in any fixed order. 

By simple observation, the phylogenetic semigroup of the right graph in Figure~\ref{figure:graph_with_three_edges} is 
\begin{displaymath}
\mathbb{N} \{(1,0,0),(1,1,0),(1,0,1)\},
\end{displaymath}
where the first coordinate corresponds to the degree and the other two coordinates correspond to edges of $G$ in any fixed order. 
\end{proof}

\begin{lemma}\label{lemma:graph2}
Let $G$ be a graph with first Betti number 2 and cycles living  in different connected components. The maximal degree of a minimal generator of $\tau(G)$ is two. 
\end{lemma}

\begin{proof}
Define $\omega \in \tau(G)$ of degree two as follows: $\omega _e=1$ for all cycle edges $e$ of a cycle $G'$ of $G$, $\omega _e=2$ for one cycle leg of $G'$, and $\omega _e=0$ for all other cycle legs of $G'$. Extend this partial labeling of $G$ in any feasible way to a degree two labeling of $G$. By Lemma~\ref{cor:degree_2_generators}, $\omega$ is indecomposable. Hence the maximal degree of a minimal generator of $\tau(G)$ is at least two.

On the other hand, we show that every element $\omega \in \tau(G)$ can be decomposed as a sum of degree one and degree two elements. By Corollary~\ref{cor:first_betti_number_1}, $\omega$ restricted to each connected component decomposes as a sum of degree one and degree two elements. These decompositions can be combined to a decomposition of $\omega\in \tau(G)$ as a sum of degree one and degree two elements. Hence the maximal degree of a minimal generator of $\tau(G)$ is exactly 2.
\end{proof}

Let $G$ be a trivalent graph and $v$ be an inner vertex of $G$. 
Every element $\omega\in\tau(G)$ decomposes locally in a unique way into paths around an inner vertex $v$. This means that there exist non-negative integers $x_v(\omega)$, $y_v(\omega)$, $z_v(\omega)$  such that
\begin{eqnarray*}
a_v(\omega)=y_v(\omega)+z_v(\omega),\\
b_v(\omega)=x_v(\omega)+z_v(\omega),\\
c_v(\omega)=x_v(\omega)+y_v(\omega), 
\end{eqnarray*}
and $x_v(\omega)+y_v(\omega)+z_v(\omega)\leq \textrm{deg}(\omega)$, see also Figure~\ref{fig_building_graph_2}. 

\begin{figure}[htbp]
\centering
\epsfxsize=120pt\epsfbox{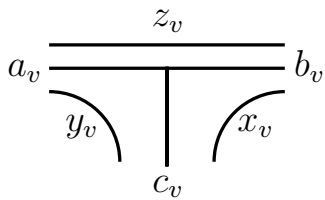}
\caption{Notation for local paths at a vertex}
\label{fig_building_graph_2}
\end{figure}

Let $T$ be a trivalent tree and $\omega_1,\omega _2\in \tau(T)$ networks. Let $v$ be an inner vertex of $T$. Then either $a_v(\omega_1)=a_v(\omega_2)$, $b_v(\omega_1)=b_v(\omega_2)$ or $c_v(\omega_1)=c_v(\omega_2)$, since $a_v(\omega_i)+b_v(\omega_i)+c_v(\omega_i)\in\{0,2\}$ for $i=1,2$. We denote this edge by $e$. By exchanging values of $\omega_1$ and $\omega_2$ on all edges of $T$ that are on the same side with $e$ from $v$, we get $\omega_1',\omega_2'\in \tau(T)$ such that $\omega_1+\omega_2=\omega_1'+\omega_2'$. We call this operation \textit{branch swapping}.

\begin{lemma}\label{lemma:graph3}
Let $G$ be a graph with first Betti number 2 containing at least one cycle leg that is not a cycle edge and two cycles sharing at least one edge. The maximal degree of a minimal generator in $\tau(G)$ is two. 
\end{lemma}

\begin{proof}
By Corollary~\ref{lemma:polygon_graph}, we can assume that $G$ is a multiple polygon graph. There is at least one cycle leg $e'$ of $G$ that is not a cycle edge for any of the cycles of $G$. Assume that $e'$ is a cycle leg of a cycle $G'$. Define $\omega $ of degree two as follows: $\omega _e=1$ for all cycle edges $e$ of $G'$, $\omega _{e'}=2$, and $\omega _e=0$ for all other edges $e$ of $G$. By Lemma~\ref{cor:degree_2_generators}, the labeling $\omega \in \tau(G)$ is indecomposable. Hence the maximal degree of a minimal generator is at least two.

On the other hand, we show that every element $\omega \in \tau(G)$ can be decomposed as a sum of degree one and degree two elements. If $G$ is not trivalent, then by~\cite[Lemma 4.1]{BBKM11} we can construct a trivalent graph $G'$ with first Betti number 2 such that the maximal degree of the minimal generating set of $\tau(G)$ is less or equal than the one of $\tau(G')$. Moreover, two cycles of $G'$ share an edge. Hence we can assume that $G$ is a trivalent graph. 

If there is a cycle edge $e$ of $G$ with $\omega_e\in\{0,\textrm{deg}(\omega)\}$, we construct the graph $G^e$ with first Betti number 1 by cutting $G$ at $e$. Denote the new leaf edges by $e_1$ and $e_2$. The labeling $\omega$ gives a labeling $\overline{\omega}$ of $G^e$. By Theorem~\ref{thm:buczynska}, the labeling $\overline{\omega}$ can be decomposed as a sum of degree one and two labelings
\begin{displaymath}
 \overline{\omega} =\sum _{i=1}^{\deg(\omega)}\overline{\omega _i},
\end{displaymath}
where
\begin{displaymath}
(\overline{\omega_i})_{e_1}=(\overline{\omega_i})_{e_2}=\left\{ 
\begin{array}{ll}
0 & \textrm{if } \omega_e=0\\
\textrm{deg}(\overline{\omega_i}) & \textrm{if } \omega_e=\textrm{deg}(\omega)
\end{array}
\right.
\end{displaymath} 
Hence the decomposition of $\overline{\omega}$ gives a decomposition of $\omega$ with all labelings having degree one or two. From now on we assume that there is no cycle edge $e$ of $G$ with $\omega_e\in\{0,\textrm{deg}(\omega)\}$.

There are exactly two vertices of $G$ incident to three cycle edges. We denote them by $u$ and $v$. We construct a tree $G'$ from $G$  by replacing the vertex $u$ with three new vertices $u_1,u_2$ and $u_3$ as in Figure~\ref{figure:replace_vertex_with_three_vertices}. The labeling $\omega$ gives a labeling $\omega'$ of $G'$. Abusing the notation slightly, we denote by $a_u(\omega'),b_u(\omega'),c_u(\omega')$ the coordinates of $\omega'$ corresponding to leaf edges with endpoints $u_1,u_2,u_3$, respectively.

\begin{figure}[ht]
\centering
\epsfxsize=200pt\epsfbox{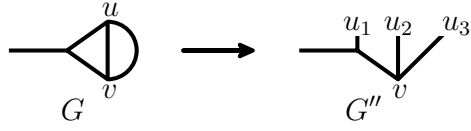} 
\caption{Construction of $G''$ from $G$ by replacing $u$ with $u_1,u_2,u_3$}
\label{figure:replace_vertex_with_three_vertices}
\end{figure}

The labeling $\omega'$ can be decomposed as a sum of degree one labelings
\begin{displaymath}
 \omega' =\sum _{i=1}^{\deg(\omega)}\omega' _i.
\end{displaymath}
From this we want to construct a decomposition of $\omega\in\tau(G)$. To lift an element of $\tau(G')$ to an element of $\tau(G)$, the parity and the degree condition have to be satisfied at leaf edges with endpoints $u_1,u_2,u_3$. This is not true for all $\omega' _i$. We need to combine and alter these elements. We will use local paths to assure the parity and degree conditions are satisfied. We will construct the decomposition of $\omega \in \tau(G)$ iteratively. In each step we construct a degree one or two element $\omega^*$ and then take $\omega:=\omega-\omega^*$.

Case 1. $\textrm{deg}_u(\omega)=\textrm{deg}(\omega)$. Note that $x_u(\omega),y_u(\omega),z_u(\omega)\geq 1$, otherwise there would be a cycle edge $e$ of $G$ with $\omega_e=\textrm{deg}(\omega)$.
\begin{itemize}
\item If there is $\omega'_i$ with exactly two of $a_u(\omega'_i),b_u(\omega'_i),c_u(\omega'_i)$ equal to 1, then $\omega'_i$ can be lifted to a degree one labeling of $G$.

\item Otherwise if there is $\omega'_i$ with exactly one of $a_u(\omega'_i),b_u(\omega'_i),c_u(\omega'_i)$ equal to 1, then there is $\omega'_j$ with all of $a_u(\omega'_j),b_u(\omega'_j),c_u(\omega'_j)$ equal to 1. Then $\omega'_i+\omega'_j$ can be lifted to a degree two labeling of $G$.

\item Otherwise there has to be $\omega'_i$ with all of $a_u(\omega'_i),b_u(\omega'_i),c_u(\omega'_i)$ equal to 0. Then there is $\omega'_j$ with all of $a_u(\omega'_j),b_u(\omega'_j),c_u(\omega'_j)$ equal to 1. After branch swapping of $\omega'_i$ and $\omega'_j$ at $v$, we get a labeling with exactly two values corresponding to $a_u,b_u,c_u$ equal to 1. It can be lifted to a degree on labeling of $G$.
\end{itemize}

Case 2. $\textrm{deg}_u(\omega)<\textrm{deg}(\omega)$. 
\begin{itemize}
\item If there exists $\omega' _i$ with $a_u(\omega'_i)=b_u(\omega'_i)=c_u(\omega'_i)=0$, then $\omega' _i$ lifts to a labeling of $\tau(G)$. 
\end{itemize}

Otherwise consider two subcases:

Case 2.1. $x_u(\omega),y_u(\omega),z_u(\omega)\geq 1$.
\begin{itemize}
\item If there is $\omega'_i$ with exactly two of $a_u(\omega'_i),b_u(\omega'_i),c_u(\omega'_i)$ equal to 1, then $\omega'_i$ can be lifted to a degree one labeling of $G$.

\item Otherwise if there is $\omega'_i$ with all of $a_u(\omega'_i),b_u(\omega'_i),c_u(\omega'_i)$ equal to 1, then there is $\omega' _j$ with exactly one of $a_u(\omega'_j),b_u(\omega'_j),c_u(\omega'_j)$ equal to 1. Then $\omega'_i+\omega'_j$ can be lifted to a degree two labeling of $G$.

\item Otherwise all $\omega'_i$ have exactly one of $a_u(\omega'_i),b_u(\omega'_i),c_u(\omega'_i)$ equal to 1. Since $x_u(\omega)\geq 1$ there is $\omega' _i$ with $a_u(\omega' _i)=c_u(\omega' _i)=0$ and $b_u(\omega' _i)=1$ and $\omega' _j$ with $a_u(\omega' _j)=b_u(\omega' _j)=0$ and $c_u(\omega' _j)=1$. Then $\omega' _i+\omega' _j$ can be lifted to a degree two labeling for $G$.
\end{itemize}

Case 2.2. Exactly two of $x_u(\omega),y_u(\omega),z_u(\omega)\geq 1$. It is not possible to have only one $x_u(\omega),y_u(\omega),z_u(\omega)\geq 1$, because we assumed $\omega_e> 0$ for every cycle edge $e$. We assume that $x_u(\omega),y_u(\omega)\geq 1$, the other two cases are analogous.
\begin{itemize}
\item If there is $\omega'_i$ with exactly $b_u(\omega'_i),c_u(\omega'_i)$ or $a_u(\omega'_i),c_u(\omega'_i)$ equal to 1, then $\omega'_i$ can be lifted to a degree one labeling of $G$.

\item Otherwise if there is $\omega'_i$ with exactly $a_u(\omega'_i),b_u(\omega'_i)$ equal to 1, there is  $\omega' _j$ with exactly $c_u(\omega'_i)$ equal to 1, since $c_u(\omega)>a_u(\omega)$ and $c_u(\omega)>b_u(\omega)$. After branch swapping $\omega' _i$ and $\omega' _j$ at $v$, we either get a labeling with all values corresponding to $a_u,b_u,c_u$ equal to 0 or a labeling with values corresponding to $b_u,c_u$ equal to 1 or a labeling with values corresponding to $a_u,c_u$ equal to 1. They all can be lifted to a degree one labeling of $G$.

\item Otherwise if there is $\omega'_i$ with all of $a_u(\omega'_i),b_u(\omega'_i),c_u(\omega'_i)$ equal to 1, there is $\omega' _j$ with exactly $c_u(\omega'_j)$ equal to 1, since $c_u(\omega)>a_u(\omega)$ and $c_u(\omega)>b_u(\omega)$. Then $\omega_i+\omega_j$ can be lifted to a degree two labeling of $G$.

\item Otherwise all $\omega'_i$ have exactly one of $a_u(\omega'_i),b_u(\omega'_i),c_u(\omega'_i)$ equal to 1. Since $x_u(\omega)\geq 1$ there is $\omega' _i$ with $a_u(\omega' _i)=c_u(\omega' _i)=0$ and $b_u(\omega' _i)=1$, and $\omega' _j$ with $a_u(\omega' _j)=b_u(\omega' _j)=0$ and $c_u(\omega' _j)=1$. Then $\omega' _i+\omega' _j$ can be lifted to a degree two labeling for $G$.

\end{itemize}

At each step a degree one or two element is constructed. This assures that the iterative process comes to an end, because the degree of $\omega$ decreases.
\end{proof}

\begin{lemma}\label{lemma:graph3b}
Let $G$ be a graph with first Betti number 2 containing at least one cycle leg that is not a cycle edge and two cycles sharing exactly one vertex. The maximal degree of a minimal generator in $\tau(G)$ is two.  
\end{lemma}

\begin{proof}
By Corollary~\ref{lemma:polygon_graph}, we can assume that $G$ is a multiple polygon graph. There is at least one cycle leg $e'$ of $G$ that is not a cycle edge for any of the cycles of $G$. Assume that $e'$ is a cycle leg of a cycle $G'$. Define $\omega $ of degree two as follows: $\omega _e=1$ for all cycle edges $e$ of $G'$, $\omega _{e'}=2$ and $\omega _e=0$ for all other edges $e$ of $G$. By Lemma~\ref{cor:degree_2_generators}, the labeling $\omega \in \tau(G)$ is indecomposable. Hence the maximal degree of a minimal generator is at least two.

On the other hand, we show that every element $\omega \in \tau(G)$ can be decomposed as a sum of degree one and degree two elements. We construct a trivalent graph $G'$ from $G$ as in~\cite[Lemma 4.1]{BBKM11} such the that the maximal degree of the minimal generating set of $\tau(G)$ is less or equal than the one of $\tau(G')$.
In particular, first we decrease the valency at the vertex $v$ that is on both cycles. We replace it by vertices $v',v''$ and an edge $e$ between them such that $e$ belongs to both cycles. We repeat replacing vertices until there are only trivalent vertices left. The graph $G'$ has two cycles that share at least one edge, thus we can apply Lemma~\ref{lemma:graph3}.  
\end{proof}

\begin{lemma}\label{lemma:graph4}
Let $G$ be a graph with first Betti number 2 where the two cycles are separated by a single edge $e$. The maximal degree of a minimal generator of $\tau(G)$ is two. 
\end{lemma}

\begin{proof}
Define $\omega \in \tau(G)$ of degree two as follows: $\omega _e=1$ for all cycle edges $e$, $\omega _e=2$ for the single edge separating cycles, and $\omega _e=0$ for all other edges. By Lemma~\ref{cor:degree_2_generators}, $\omega$ is indecomposable. Hence the maximal degree of a minimal generator of $\tau(G)$ is at least two.

On the other hand, we show that every element $\omega \in \tau(G)$ can be decomposed as a sum of degree one and degree two elements. If $G$ is not trivalent, then by~\cite[Lemma 4.1]{BBKM11} we can construct a trivalent graph $G'$ with first Betti number 2 such that the maximal degree of the minimal generating set of $\tau(G)$ is less or equal than the maximal degree of the minimal generating set of $\tau(G')$. Moreover, we may assume that every time we replace a vertex $v$ on a cycle by vertices $v',v''$ and an edge between them, then $v',v''$ belong to the same cycle. This assures that the  two cycles of $G'$ are separated by a singe edge. Hence we can assume that $G$ is a trivalent graph.

Let $e_1,e_2$ be new leaf edges obtained by cutting $G$ at $e$ and write $G^e=G_1\sqcup G_2$. The labeling $\omega$ gives labelings $\omega_1$ of $G_1$ and $\omega_2$ of $G_2$. By Corollary~\ref{cor:first_betti_number_1}, we can decompose $\omega_1$ and $\omega_2$ as a sum of degree one and degree two elements. Because all degree two labelings in these decompositions have values 0 or 2 corresponding to the edges $e_1$ and $e_2$, we can combine decompositions of $\omega_1$ and $\omega_2$ to get a decomposition of $\omega$ that consists of degree one and two elements. Hence the maximal degree of a minimal generator of $\tau(G)$ is exactly two.
\end{proof}

\begin{lemma}\label{lemma:graph5}
Let $G$ be a graph with first Betti number 2 where the two cycles are separated by at least one inner vertex. The maximal degree of a minimal generator of $\tau(G)$ is three.
\end{lemma}

\begin{proof}
By Corollary~\ref{lemma:polygon_graph}, we can assume that $G$ is a multiple polygon graph. We need to specify a degree three indecomposable element $\omega \in \tau(G)$. Fix an inner vertex $v$ on the path between the two cycles of $G$ and an edge $e^*$ incident to $v$ that is not on the path between the two cycles. Define $\omega _e=2$ for all cycle edges $e$ and all edges $e$ on the path between the cycles of $G$, $\omega _{e^*}=2$, and $\omega _e=0$ for all other edges $e$. 

We will show that $\omega $ is indecomposable as a degree three labeling. By contradiction, assume $\omega=\omega_1+\omega _2$, where deg$(\omega_1)=1$ and deg$(\omega_2)=2$. We must have $(\omega_2)_e=1$ for all cycle edges of $G$ and $(\omega_2)_e=2$ for both cycle legs $e$ that lie on the path between the two cycles. Hence also $(\omega _2)_e=2$ for all edges $e$ that lie on the path between the two cycles. Thus  $(\omega _1)_{e^*}=2$ . This leads to a contradiction, because deg$(\omega _1)=1$. Hence $\omega$ is a degree three indecomposable element in $\tau(G)$.
\end{proof}

\begin{proof}[Proof of Theorem~\ref{thm:graphs_with_first_betti_number_2}]
Theorem follows from Lemmas~\ref{lemma:graph1}--\ref{lemma:graph5}.
\end{proof}

\section{Degree Three Minimal Generators}\label{section:degree_three_generators}

In this section, we describe degree three minimal generators of phylogenetic semigroups on trivalent graphs with first Betti number 2 having at least one inner vertex between two cycles. Together with  Lemma~\ref{lemma:networks_on_graphs}, Proposition~\ref{prop:degree_2_generators} and Theorem~\ref{thm:graphs_with_first_betti_number_2} this completely characterizes minimal generating sets of phylogenetic semigroups on trivalent graphs with first Betti number 2.

Let $G$ be a trivalent graph with first Betti number 1 and $\omega$ a degree two indecomposable labeling. Label the cycle legs of $G$ where $\omega$ has value two by $e_0,\ldots ,e_{2k}$ in clockwise order. Slightly abusing notation, we write $e_{i+j}$ for $e_{i+j \mod 2k}$ where $i+j>2k$. Label by $P_{e',e''}$  the path starting at a cycle leg $e'$ and going in the clockwise direction until reaching a cycle leg $e''$ . Write $P_i$ for $P_{e_i,e_{i+1}}$.  We say a cycle leg $e$ is between cycle legs $e_i$ and $e_j$, when $e$ is between cycle legs $e_i$ and $e_j$ in clockwise direction.

\begin{lemma}\label{lemma:decompositions_of_degree_three_labelings}
Let $G$ be a trivalent polygon graph and $\omega \in \tau(G)$ a degree three labeling. Then $\omega $ cannot be decomposed as a sum of degree one labelings if and only if $\omega =\omega _1+\omega_2$ such that
\begin{itemize}
\item $\deg(\omega_1)=1$ and $\deg(\omega_2)=2$,
\item $\omega_2$ is indecomposable with value two on cycle legs $e_0,\ldots ,e_{2k}$ and value zero on all other cycle legs,
\item $\omega_1$ is $P_0\cup P_2 \cup \ldots \cup P_{2k-2}$, $P_0 \cup P_2 \cup \ldots P_{2k-4} \cup P_{e_{2k-2},e_{2k}}$, $P_{0}\cup P_2 \cup \ldots P_{2k-2} \cup P_{e_{2k},e'}$ where $e'$ is a cycle leg between $e_{2k}$ and $e_{0}$ or also the cycle path if $k=0$.
\end{itemize}
\end{lemma}

\begin{proof}
A degree three labeling $\omega$ can be always decomposed as $\omega=\omega_1+\omega_2$ with $\deg(\omega_1)=1$ and $\deg(\omega_2)=2$. We show that unless $\omega_1,\omega_2$ are as in the statement of the lemma, we can alter $\omega_1,\omega_2$ to get $\omega'_1,\omega'_2$ such that $\omega_1+\omega_2=\omega'_1+\omega'_2$ and $\omega'_2$ decomposes as a sum of two degree one labelings. 

If there is $P_i$ such that it does not intersect the network $\omega_1$, then the union of $\omega_1$ and $P_i$ is a network and the complement of $P_i$ in $\omega_2$ decomposes as the sum of $P_{i+1}\cup P_{i+3} \cup \ldots \cup P_{i-2}$ and $P_{i+2}\cup P_{i+4} \cup \ldots \cup P_{i-1}$. We will assume from now on that every $P_i$ intersects $\omega_1$.

If there exist $e_i$ and $e_j$ such that neither of them is incident to a path in $\omega_1$, then either $e_j=e_{i+2l}$  or $e_i=e_{j+2l}$ for some $1\leq l \leq k$. In the first case, let $\Gamma$ be the union of paths in $\omega_1$ from $e_i$ to $e_j$ and define $\omega'_1=\omega_1 \backslash \Gamma \cup P_{i} \cup P_{i+2} \cup \ldots \cup P_{j-2}$ and $\omega'_2=\omega_2 \backslash (P_{i} \cup P_{i+2} \cup \ldots \cup P_{j-2}) \cup \Gamma$.  Then $\omega'_2$ decomposes as the sum $\Gamma\cup P_{j} \cup P_{j+2} \cup \ldots \cup P_{i-1}$ and $P_{j+1}\cup P_{j+3}\cup \ldots \cup P_{i-2}\cup P_{i+1} \cup P_{i+3} \cup \ldots \cup P_{j-1}$.  In the second case, the same discussion applies for $i$ and $j$ exchanged. We will assume from now on that there is at most one $e_i$ that is not incident to a path in $\omega_1$.

If $\omega _1$ corresponds to the cycle path $P_{\textrm{cycle}}$ and $k\geq 1$, then $\omega$ decomposes as the sum of $P_{e_2,e_1}$ and $P_{e_0,e_2}\cup P_3 \cup P_5 \cup \ldots \cup P_{2k-1}$ and $P_{e_1,e_3}\cup P_4 \cup P_6 \cup \ldots \cup P_{2k}$. Here we use that $P_{\textrm{cycle}}\cup P_0 \cup P_1 \cup P_2 = P_{e_0,e_2} \cup P_{e_1,e_3} \cup P_{e_2,e_1}$.

It there is a path $P_{e',e''}$  in $\omega_1$ such that not both $e',e''$ belong to $\{e_0,\ldots ,e_{2k}\}$, then we consider five different cases: 
\begin{itemize}
\item If there is a path $P_i$  such that $e',e''$ are both between $e_i$ and $e_{i+1}$, define $\omega_1'=\omega_1\backslash P_{e',e''}$ and $\omega_2''=\omega_2\cup P_{e',e''}$. Since $P_{e',e''}\cup P_i=P_{e_i,e''}\cup P_{e',e_{i+1}}$, the labeling $\omega'_2$ decomposes as the sum of $P_{e_i,e''}\cup P_{i+1}\cup P_{i+3}\cup \ldots \cup P_{i-2}$ and $P_{e',e_{i+1}}\cup P_{i+2}\cup P_{i+4}\cup \ldots \cup P_{i-1}$. 

\item If there is a path $P_{i}$ such that $e'$ is before $e_i$ and $e''$ is between $e_{i}$ and $e_{i+1}$, define $\omega'_1=\omega_1\backslash P_{e',e''} \cup P_{e',e_i}$ and $\omega'_2=\omega_2 \backslash (P_{i-1}\cup P_i) \cup P_{e_{i-1},e_{i+1}} \cup P_{e_i,e''}$. Then $\omega_1+\omega_2=\omega'_1+\omega'_2$, since $P_{e',e''}\cup P_{i-1}\cup P_i=P_{e',e_i}\cup P_{e_{i-1},e_{i+1}} \cup P_{e_i,e''}$. Then $\omega'_2$ decomposes as the sum of $P_{e_{i-1},e_{i+1}}\cup P_{i+2},\ldots \cup P_{{i-3}}$ and $P_{e_i,e''}\cup P_{i+1} \cup \ldots \cup P_{i-2}$. 

\item If there are paths $P_{e_i,e'}$ and $P_{e'',e_{j}}$ in $\omega_1$ such that $e'$ is between $e_i$ and $e_{i+1}$, and $e''$ is between $e_{j-1}$ and $e_{j}$ and $e_j=e_{i+2l+1}$ for some $0\leq l\leq k-1$, let $\Gamma$ be the union of paths between $e_i$ and $e_j$.  Define $\omega'_1=\omega_1\backslash \Gamma \cup P_i \cup P_{i+2}\cup \ldots \cup P_{j-1}$ and $\omega'_2=\omega_2 \backslash (P_i \cup P_{i+2}\cup \ldots \cup P_{j-1}) \cup  \Gamma$. Then $\omega'_2$ decomposes as the sum of $P_{e_i,e'}\cup P_{i+1} \cup P_{i+3} \cup \ldots \cup P_{i-2}$ and $\Gamma \backslash P_{e_i,e'} \cup P_{j+1} \cup P_{j+3} \cup \ldots \cup P_{i-1}$.

\item If there are paths $P_{e_i,e'}$ and $P_{e'',e_{j}}$ in $\omega_1$ such that the edge $e'$ is between $e_i$ and $e_{i+1}$, the edge $e''$ is between $e_{j-1}$ and $e_{j}$ and $e_j=e_{i+2l}$ for some $1\leq l\leq k$, let $\Gamma$ be the union of paths between $e_j$ and $e_i$ together with $P_{e_i,e'}$ and $P_{e'',e_{j}}$. Define $\omega'_1=\omega_1\backslash \Gamma \cup P_j \cup P_{j+2} \cup \ldots \cup P_{i-1}$ and $\omega'_2=\omega_2\backslash (P_j \cup P_{j+2} \cup \ldots \cup P_{i-1}) \cup \Gamma$. Then $\omega'_2$ decomposes as the sum of $\Gamma \backslash P_{e_i,e'} \cup P_i \cup P_{i+2} \cup \ldots \cup P_{j-2}$ and $P_{e_i,e'}\cup P_{i+1} \cup P_{i+3} \cup \ldots \cup P_{i-2}$.

\item If there are paths $P_{e_i,e'}$ and $P_{e_j,e''}$ in $\omega _1$ such that the edge $e'$ is between $e_i$ and $e_{i+1}$, the edge $e''$ is between $e_j$ and $e_{j+1}$ and $e_j=e_{i+2l}$ for some $1\leq l\leq k$, let $\Gamma$ be the union of paths in $\omega_1$ between $e_i$ and $e''$ without $P_{e_i,e'}$. Define $\omega'_1=\omega_1 \backslash \Gamma \cup P_{i+1} \cup P_{i+3} \cup \ldots \cup P_{j-1}$ and $\omega'_2=\omega_2\backslash (P_{i+1} \cup P_{i+3} \cup \ldots \cup P_{j-1}) \cup \Gamma$. Then $\omega'_2$ decomposes as the sum of $P_{e_j,e''}\cup P_{j+1} \cup P_{j+3} \cup \ldots \cup P_{j-2}$ and $P_{j} \cup P_{j+2} \cup \cdots \cup P_{i-1} \cup \Gamma \backslash P_{e_j,e''}$. If $e_j=e_{i+2l+1}$ for some $0\leq l\leq k-1$ then the same discussion works for $i$ and $j$ exchanged. 
\end{itemize}
If none of the five if-conditions holds, then the unique path of the form $P_{e',e''}$  in $\omega_1$ such that not both $e',e''$ belong to $\{e_0,\ldots ,e_{2k}\}$ must be  $P_{e_i,e'}$ or $P_{e'',e_{i+1}}$ with $e,'e''$ between $e_i$ and $e_{i+1}$.  

If there is a path in $\omega_1$ of the form $P_{e_i,e_j}$, then we consider three different cases:
\begin{itemize}
\item If there exists $P_{e_i,e_j}$ with $e_j=e_{i+2l+1}$ for $1\leq l \leq k-1$, then define $\omega'_1=\omega_1\backslash P_{e_i,e_j} \cup P_i \cup P_{i+2} \cup \ldots \cup P_{j-1}$ and $\omega'_2=\omega_2\backslash (P_i \cup P_{i+2} \cup \ldots \cup P_{j-1}) \cup P_{e_i,e_j}$. Since $P_{e_i,e_j}\cup P_{j-2}=P_{e_i,e_{j-1}}\cup P_{e_{j-2},e_j}$, then $\omega'_2$ decomposes as the sum of $P_{e_i,e_{j-1}}\cup P_j \cup P_{j+2} \cup \ldots \cup P_{i-2}$ and $P_{i+1} \cup P_{i+3} \cup \ldots \cup P_{j-4}\cup P_{e_{j-2},e_j}\cup P_{j+1} \cup P_{j+3} \cup \ldots \cup P_{i-1}$.

\item If there exists $P_{e_i,e_j}$ with $e_j=e_{i+2l}$ for $2\leq l \leq k$, assume that $j=0$. Define $\omega'_1=\omega_1 \backslash P_{e_i,e_j} \cup P_{e_i,e_{j-2}} \cup P_{j-1}$ and $\omega'_2=\omega_2 \backslash (P_{j-3} \cup P_{j-2} \cup P_{j-1}) \cup P_{j-3,j-1} \cup P_{j-2,j}$. Then $\omega_1+\omega_2=\omega'_1+\omega'_2$ since $P_{e_i,e_j}\cup P_{j-3}\cup P_{j-2} \cup P_{j-1}=P_{e_i,e_{j-2}} \cup P_{j-3,j-1} \cup P_{j-2,j}\cup P_{j-1}$. Then $\omega'_2$ decomposes as the sum of $P_0 \cup P_2 \cup \ldots \cup P_{j-3,j-1}$ and $P_1 \cup P_3 \cup \ldots \cup P_{j-2,j}$.

\item If there exist $P_{e_{i_1},e_{i_2}},P_{e_{j_1},e_{j_2}}$ such that $e_{i_2}=e_{i_1+2l},e_{j_2}=e_{j_1+2m}$ for $1\leq l,m \leq k$ and $e_{j_2}=e_{i_1+2n+1}$ for some $0\leq n\leq k-1$, assume that $P_{j_1,j_2}$ is the next path with such property after $P_{i_1,i_2}$ in clockwise direction. Denote all paths between $e_{i_1}$ and $e_{j_2}$ in $\omega_1$ by $\Gamma$. Define $\omega'_1=\omega_1 \backslash \Gamma \cup P_{i_1} \cup P_{i_1+2} \cup \ldots \cup P_{j_2-1}$ and $\omega'_2=\omega_2 \backslash (P_{i_1} \cup P_{i_1+2} \cup \ldots \cup P_{j_2-1}) \cup \Gamma$. Then $\omega'_2$ decomposes as the sum of $\Gamma \backslash \ P_{e_{j_1},e_{j_2}} \cup P_{j_1} \cup P_{j_1+2} \cup \cdots \cup P_{i_1-2}$ and $P_{i_1+1} \cup P_{i_1+3} \cup \ldots \cup P_{j_1-2} \cup P_{e_{j_1},e_{j_2}} \cup P_{j_2+1} \cup \ldots \cup P_{i_1-1}$. If $e_{j_2}=e_{i_1+2n}$ for some $1\leq n\leq k-1$ then the same discussion works for $i$'s and $j$'s exchanged.
\end{itemize}

If none of the three if-conditions holds, then only paths in $\omega_1$ of the form $P_{e_i,e_j}$ can be $P_i$ and at most one $P_{e_j,e_{j+2}}$.

Finally we have to show that $\omega_1$ cannot simultaneously contain paths $P_{e_i,e'}$ and $P_{e_j,e_{j+2}}$ where $i,j\in\{0,\ldots ,2k\}$ and $e'$ is between $e_i$ and $e_{i+1}$. We consider two different cases: $e_j=e_{i+2l}$ for $1\leq l \leq k-1$ and $e_j=e_{i+2l+1}$ for $0\leq l \leq k-1$. If we have $P_{e',e_i}$ instead of $P_{e_i,e'}$ then we can apply the same discussion in the counterclockwise direction.

In the first case there must be $e_t$ with $t=i+2l+1$ between $e'$ and $e_j$ that is not incident to any of the paths in $\omega_1$. Otherwise paths between $e'$ and $e_j$ in $\omega_1$ would be $P_{i+1},P_{i+3},\ldots,P_{j-1}$, which is not possible, since $P_{e_j,e_{j+2}}$ is in $\omega _1$. Let $\Gamma$ be the union of paths in $\omega_1$ between $e_i$ and $e_t$. Define $\omega'_1=\omega_1 \backslash \Gamma \cup P_{i} \cup P_{i+2} \cup \ldots \cup P_{t-1}$ and $\omega'_2=\omega_2 \backslash (P_{i} \cup P_{i+2} \cup \ldots \cup P_{t-1}) \cup \Gamma$. Then $\omega'_2$ decomposes as the sum of $\Gamma \cup P_t \cup P_{t+2} \cup \ldots P_{i-2}$ and $P_{t+1}\cup P_{t+3} \cup \ldots P_{t-2}$. 

In the second case let $\Gamma$ be the union of paths in $\omega_1$ from $e_i$ to $e_{j+2}$. Define $\omega'_1=\omega_1 \backslash \Gamma \cup P_i \cup P_{i+2} \cup \ldots \cup P_{j+1}$ and $\omega'_2=\omega_2 \backslash (P_i \cup P_{i+2} \cup \ldots \cup P_{j+1}) \cup \Gamma$. Then $\omega'_2$ decomposes as the sum of $P_{e_i,e'}\cup P_{i+1} \cup P_{i+3} \cup \ldots \cup P_{i-2}$ and $\Gamma \backslash P_{e_i,e'} \cup P_{j+3} \cup P_{j+5} \cup \ldots \cup P_{i-1}$.

If none of the previous is true for $\omega$, then $\omega$ is as in the statement of the lemma. We will show that the only decomposition of $\omega$ is $\omega=\omega_1+\omega_2$. This implies that we cannot decompose $\omega$ as a sum of degree one labelings as $\omega_2$ is indecomposable.  

In all four cases, we have $\omega_e\geq 2$ for a cycle leg $e$ if and only if $e\in \{e_0,\ldots ,e_{2k}\}$. Moreover, $\omega_{\overline{e}}=2$ holds for at most one $\overline{e}\in \{e_0,\ldots ,e_{2k}\}$. We construct a decomposition $\omega=\omega'_1+\omega'_2$ with $\deg(\omega'_1)=1$ and $\deg(\omega'_2)=2$. If $\omega_1=P_{0}\cup P_2 \cup \ldots P_{2k-2} \cup P_{e_{2k},e'}$, then $(\omega'_2)_e=2$ for all $e\in \{e_0,\ldots ,e_{2k}\}$ and $(\omega'_2)_e=0$ for all other cycle legs. Indeed, there is only one cycle leg $e'$ left with value one, but since the sum of values on all leaf edges must be even we have $(\omega'_2)_{e'}=0$. For the other three cases $(\omega'_2)_e=2$ for all $e\in \{e_0,\ldots ,e_{2k}\}\backslash{\overline{e}}$, since $\omega_e=3$ for  $e\in \{e_0,\ldots ,e_{2k}\}\backslash{\overline{e}}$. Thus also $(\omega'_2)_{\overline{e}}=2$, since the sum of values on all leaf edges must be even. It follows that $(\omega'_2)_e=1$ for all cycle edges, hence $\omega'_i=\omega_i$ for $i\in\{1,2\}$.
\end{proof}

\begin{corollary}\label{cor:unique_decomposition}
Let $G$ be a polygon graph and $\omega \in \tau(G)$ a degree three labeling. Then $\omega $ cannot be decomposed as a sum of degree one labelings if and only if $\omega$ can be decomposed uniquely as $\omega =\omega _1+\omega_2$ with $\deg(\omega_1)=1$ and $\deg(\omega_2)=2$.
\end{corollary}

\begin{corollary}
Let $G$ be a trivalent graph with first Betti number 1 and $\omega \in \tau(G)$ a degree three labeling. Then $\omega $ cannot be decomposed as a sum of degree one labelings if and only if $\omega =\omega _1+\omega_2$ such that
\begin{itemize}
\item $\deg(\omega_1)=1$ and $\deg(\omega_2)=2$,
\item $\omega_2$ is indecomposable with value two on cycle legs $e_0,\ldots ,e_{2k}$ and value zero on all other cycle legs,
\item $\omega_1$ restricted to the unique cycle with its cycle legs is $P_0\cup P_2 \cup \ldots \cup P_{2k-2}$, $P_0 \cup P_2 \cup \ldots P_{2k-4} \cup P_{e_{2k-2},e_{2k}}$, $P_{0}\cup P_2 \cup \ldots P_{2k-4} \cup P_{e_{2k-2},e'}$ where $e'$ is a cycle leg between $e_{2k-2}$ and $e_{2k-1}$ or also the cycle path if $k=0$.
\end{itemize}
\end{corollary}

\begin{proof}
The statement follows directly from Lemma~\ref{lemma:decompositions_of_degree_three_labelings} and~\ref{lemma:polygon_graph}. 
\end{proof}

\begin{lemma}\label{lemma:degree_three_generators}
Let $G$ be a trivalent graph with first Betti number 2 where the two cycles are separated by at least one inner vertex and $\omega \in \tau(G)$ a degree three labeling. Then $\omega $ is indecomposable if and only if the following conditions are fulfilled:
\begin{itemize}
\item[(i)] $\omega$ restricted to any cycle with its cycle legs does not decompose as a sum of degree one labelings, 

\item[(ii)] $\omega$ restricted to an edge on the shortest path between two cycles has value one or two,

\item[(iii)] $\omega$ restricted to exactly one edge incident to an edge on the shortest path between two cycles that is not a cycle edge or an edge on the shortest path has value one or two, and has value zero or three on all other such edges.
\end{itemize}
\end{lemma}

\begin{proof}
By Lemma~\ref{lemma:polygon_graph}, we can assume that $G$ is a multiple polygon graph. Depict the edges on the shortest path between the two cycles horizontally and edges incident to them vertically below them as in Figure~\ref{figure:graphs_with_first_betti_number_2_cases} $(c)$.

Assume $\omega $ restricted to a cycle $G_1$ together with its cycle legs decomposes as a sum of degree one elements. Let $e$ be a cycle leg of $G_1$ on the shortest path between $G_1$ and the other cycle. Write $G^e=G_1\sqcup G_2$. Then $\omega $ decomposes on $G_2$, and this decomposition can be extended to $G$.

Assume there is an edge $\overline{e}$ on the shortest path between two cycles of $G$ such that $\omega _{\overline{e}}\in\{0,3\}$. Let $e',e''$ be new leaf edges obtained by cutting $G$ at $e$ and write $G^{\overline{e}}=G'\sqcup G''$. Then $\omega |_{G'}$ and $\omega |_{G''}$ can be decomposed as
\begin{displaymath}
 \omega |_{G'}=\omega'_1+\omega'_2 \textrm{ and } \omega |_{G''}=\omega''_1+\omega''_2
\end{displaymath}
with $\deg(\omega'_1)=\deg(\omega''_1)=1$ and $\deg(\omega'_2)=\deg(\omega''_2)=2$. Furthermore, $(\omega'_i)_{e'}=(\omega''_i)_{e''}$ for $i=1,2$ and hence they can be combined to a decomposition of $\omega$.

Assume now that the conditions $(i),(ii)$ are fulfilled. The labeling $\omega $ can be decomposed if and only if it can be decomposed as
\begin{displaymath}
 \omega =\omega_1+\omega_2
\end{displaymath}
with $\deg(\omega_1)=1$ and $\deg(\omega_2)=2$. There is a unique way of defining $\omega_1$ and $\omega_2$ on cycles and cycle legs by Corollary~\ref{cor:unique_decomposition}. We try to construct a decomposition of $\omega$ on all other edges step-by-step going from left to right such that the decomposition is compatible with the decomposition on the cycle legs on the shortest path between the two cycles, and study when there exists no such decomposition.   Let $e$ be the leftmost edge of the shortest path between two cycles where $\omega_1$ and $\omega_2$ are defined and let the vertex $v$ be the right endpoint of $e$. We want to define $b_v(\omega_i)$ and $c_v(\omega_i)$ given $a_v(\omega_i)$ for $i=1,2$.

All possible local decompositions of $\omega$ at an inner vertex between the two cycles (assuming that horizontal edges have values one or two) are presented in Figures~\ref{figure:degree_three_local_decompositions} and~\ref{figure:degree_three_local_decompositions_with_options}. In Figure~\ref{figure:degree_three_local_decompositions} the value of $\omega$ at the vertical edge is zero or three. In Figure~\ref{figure:degree_three_local_decompositions_with_options} the value of $\omega$ at the vertical edge is one or two.

\begin{figure}[ht]
\centering
\epsfxsize=280pt\epsfbox{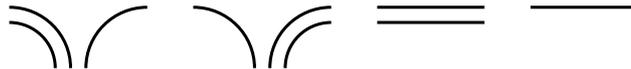}
\caption{Degree three local decompositions, 1}
\label{figure:degree_three_local_decompositions}
\end{figure}

\begin{figure}[ht]
\centering
\epsfxsize=280pt\epsfbox{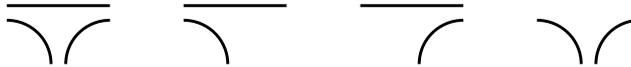}
\caption{Degree three local decompositions, 2}
\label{figure:degree_three_local_decompositions_with_options}
\end{figure}

 Given a local decomposition at $v$ as in Figure~\ref{figure:degree_three_local_decompositions} and $a_v(\omega_i)$, then there is a unique way of defining $b_v(\omega_i)$ and $c_v(\omega_i)$. In particular, if $a_v(\omega_2)\in\{0,2\}$ then $b_v(\omega_2)\in\{0,2\}$. If $a_v(\omega_2)=1$ then $b_v(\omega_2)=1$. Given a local decomposition at $v$ as in Figure~\ref{figure:degree_three_local_decompositions_with_options} and $a_v(\omega_i)$, then there might be a unique way of defining  $b_v(\omega_i)$ and $c_v(\omega_i)$ or not depending on the value of $a_v(\omega_2)$. If $a_v(\omega_2)\in\{0,2\}$ then $b_v(\omega _2)=1$. If $a_v(\omega_2)=1$ then one can define either $b_v(\omega _2)\in \{0,2\}$ or $b_v(\omega _2)=1$.

Let $e$ be a cycle leg  that is on the path between two cycles. If $\omega _{e}=2$, then $(\omega_1)_{e}=0$ and $(\omega_2)_{e}=2$, because a degree two indecomposable element on a cycle can have only values zero and two on cycle legs by Theorem~\ref{thm:buczynska}. If $\omega _{e}=1$, then $(\omega_{e})_1=1$ and $(\omega_{e})_2=0$ for the same reasons. Denote by $e_r$ the cycle leg of the right cycle that are on the path between two cycles.

If the horizontal path contains labelings only as in Figure~\ref{figure:degree_three_local_decompositions}, then $b_v(\omega_2)\in\{0,2\}$ for every vertex $v$ on the horizontal path. In particular, $(\omega_2)_{e_r}\in\{0,2\}$, hence there exists a decomposition of $\omega$.

If at more than one vertex the local decomposition is as in Figure~\ref{figure:degree_three_local_decompositions_with_options}, denote the first such vertex by $v'$ and the last one by $v''$. For all the vertices $v$ left from $v'$ the value $b_v(\omega)\in\{0,2\}$ is uniquely defined. For $v'$ we have $b_{v'}(\omega)=1$. For all the vertices $v$ between $v'$ and $v''$ we can define $b_v(\omega)=1$: If the local decomposition at $v$ is as in Figure~\ref{figure:degree_three_local_decompositions_with_options} then we have this choice by the discussion below. If the local decomposition at $v$ is as in Figure~\ref{figure:degree_three_local_decompositions} then $b_v(\omega_2)=1$ since $a_v(\omega_2)=1$ again by the discussion above. For $v=v''$ define $b_v(\omega_2)\in\{0,2\}$. At all vertices $v$ to the right of $v''$, we have local decompositions as in Figure~\ref{figure:degree_three_local_decompositions}, therefore $b_v(\omega_2)\in\{0,2\}$. In particular, $(\omega_2)_{e_r}\in\{0,2\}$ and the decomposition of $\omega$ on the horizontal path is compatible with the decompositions of $\omega$ on both cycles. Hence $\omega $ is decomposable. 

On the other hand, if at one vertex $v'$ the local decomposition is as in Figure~\ref{figure:degree_three_local_decompositions_with_options} and at all other vertices the local decomposition is as in Figure~\ref{figure:degree_three_local_decompositions}, then $b_v(\omega_2)\in\{0,2\}$ for all vertices $v$ left from $v'$ and $b_v(\omega_2)=1$ for all vertices $v$ to the right of $v'$ including $v'$ itself. In particular, $(\omega_2)_{e_r}=1$ which is not compatible with the values of $\omega_2$ on the right cycle. Since all steps have been uniquely determined, then $\omega$ cannot be decomposed. This completes the proof.
\end{proof}

\section{Examples}\label{section:examples_of_maximal_degrees}

In this section, we will list some examples of graphs with first Betti number 3, 4 and 5 together with the maximal degree of the minimal generating set of their phylogenetic semigroup.\footnote{We thank Christopher Manon for introducing us the trivalent graph with first Betti number 4 and maximal degree one, see Figure~\ref{figure:first_betti_number_four}.} Maximal degrees have been computed with \texttt{Normaliz}~\cite{normaliz}. We will also show that for any natural number $g$ there exists a graph $G$ with first Betti number $g$ such that the maximal degree of a minimal generator of $\tau(G)$ is one. We note that the maximal degree tends to depend on the ``separateness'' of the cycles, exactly as for graphs with first Betti number 2.

\begin{example}
Let $G$ be the graph with first Betti number $g$ that has two vertices and $g+1$ edges between the two vertices, as illustrated in Figure~\ref{figure:two_vertices_graph}. Then $\tau(G)$ is generated in degree one. By cutting all edges of $G$, we get two claw trees $T',T''$ with $g+1$ leaves. Let $\omega\in \tau(G)$ be a degree $d$ labeling. Then $\omega$ gives $\omega'\in T'$ and $\omega'' \in T''$ with $\omega'=\omega''$ that we can decompose as a sum of $d$ degree one labelings exactly the same way on both trees. Gluing the decompositions of $\omega'$ and $\omega''$ gives a decomposition of $\omega$ as a sum of degree one labelings.
\end{example}

\begin{figure}[ht]
\centering
\epsfxsize=80pt\epsfbox{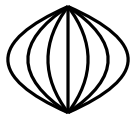}
\caption{Graph with two vertices and seven edges}
\label{figure:two_vertices_graph}
\end{figure} 

\begin{figure}[H]
\centering
\epsfxsize=400pt\epsfbox{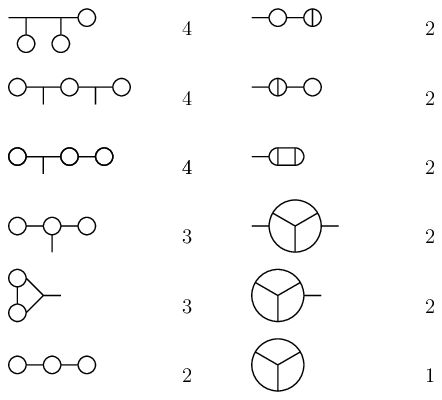}
\caption{Maximal degrees of the minimal generating set of $\tau(G)$ where $G$ is a graph with first Betti number 3}
\end{figure}

\begin{figure}[H]
\centering
\epsfxsize=400pt\epsfbox{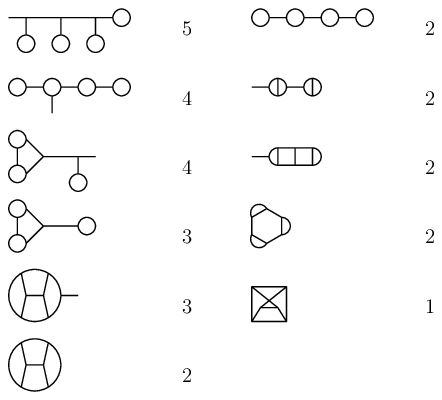}
\caption{Maximal degrees of the minimal generating set of $\tau(G)$ where $G$ is a graph with first Betti number 4}

\label{figure:first_betti_number_four}
\end{figure}

\begin{figure}[H]
\centering
\epsfxsize=400pt\epsfbox{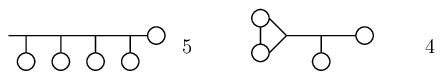}
\caption{Maximal degrees of the minimal generating set of $\tau(G)$ where $G$ is a graph with first Betti number 5}
\end{figure}

\section*{Acknowledgements}

We thank Jaroslaw Buczy{\'n}ski, Christopher Manon and Henning Meyer for their helpful comments.

\bibliographystyle{alpha}
\bibliography{low_degree_minimal_generators_of_phylogenetic_semigroups}

\end{document}